\documentclass[11pt]{amsart}
\pdfoutput=1

\usepackage{amssymb}
\usepackage{amsthm}
\usepackage{amsfonts}
\usepackage{amsmath}
\usepackage{pmboxdraw}
\usepackage{verbatim} 
\usepackage{graphicx}
\usepackage{color}
\usepackage[colorlinks=true, citecolor=blue, filecolor=black, linkcolor=black, urlcolor=black]{hyperref}
\usepackage{cite}
\usepackage[normalem]{ulem}
\usepackage{subcaption}
\usepackage{bbm}
\usepackage{bm}
\usepackage{mathtools}
\usepackage{enumitem}
\mathtoolsset{showonlyrefs}
\usepackage{todonotes}

\usepackage{amsmath}
\usepackage{xparse}

\ExplSyntaxOn
\NewDocumentCommand{\MeijerG}{smmmm}
 {
  \IfBooleanTF{#1}
   {
    \vic_meijerg:nnnnnn { #2 } { #3 } { #4 } { #5 } { small } { }
   }
   {
    \vic_meijerg:nnnnnn { #2 } { #3 } { #4 } { #5 } { } { \; }
   }
 }

\seq_new:N \l__vic_meijerg_args_in_seq
\seq_new:N \l__vic_meijerg_args_out_seq

\cs_new_protected:Nn \vic_meijerg:nnnnnn
 {
  \seq_set_split:Nnn \l__vic_meijerg_args_in_seq { | } { #3 }
  \seq_clear:N \l__vic_meijerg_args_out_seq  
  \seq_map_inline:Nn \l__vic_meijerg_args_in_seq
   {
    \seq_put_right:Nn \l__vic_meijerg_args_out_seq
     {
      \begin{#5matrix} ##1 \end{#5matrix}
     }
   }
  G\sp{#1}\sb{#2}
  \left(
  \seq_use:Nn \l__vic_meijerg_args_out_seq { #6\middle|#6 }
  #6\middle|#6
  #4
  \right)
 }
\ExplSyntaxOff

\newcommand{\RN}[1]{%
	\textup{\uppercase\expandafter{\romannumeral#1}}%
}

\def\pa{\partial}

\def\wt{\widetilde}

\def\C{\mathbb{C}}

\def\R{\mathbb{R}}

\newcommand{\erfc}{\operatorname{erfc}}
\newcommand{\erf}{\operatorname{erf}}
\newcommand{\bfR}{\mathbf{R}}

\newcommand{\sgn}{\operatorname{sgn}}

\theoremstyle{plain}
\newtheorem*{thm*}{Theorem}
\newtheorem{thm}{Theorem}[section]
\newtheorem{lem}[thm]{Lemma}
\newtheorem{cor}[thm]{Corollary}
\newtheorem{prop}[thm]{Proposition}
\newtheorem*{prop*}{Proposition}
\newtheorem*{lem*}{Lemma}

\theoremstyle{definition}
\newtheorem*{eg*}{Example}
\newtheorem*{egs*}{Examples}
\newtheorem*{def*}{Definition}
\newtheorem*{Q*}{Question}

\theoremstyle{remark}
\newtheorem*{rmk*}{Remark}
\newtheorem*{rmks*}{Remarks}

\numberwithin{equation}{section}

\begin{document}
\title[Real eigenvalues of products of real Ginibre matrices]{The Product of $m$ real $N\times N$ Ginibre matrices: \\
Real eigenvalues in the critical regime $m=O(N)$}

\author{Gernot Akemann}
\address{Faculty of Physics, Bielefeld University, P.O. Box 100131, 33501 Bielefeld, Germany}
\email{akemann@physik.uni-bielefeld.de}

\author{Sung-Soo Byun}
\address{School of Mathematics, Korea Institute for Advanced Study, 85 Hoegiro, Dongdaemun-gu, Seoul 02455, Republic of Korea}
\email{sungsoobyun@kias.re.kr}


\thanks{The work of Gernot Akemann was partly funded by the Deutsche Forschungsgemeinschaft (DFG) grant SFB 1283/2 2021 – 317210226.  
Sung-Soo Byun was partially supported by Samsung Science and Technology Foundation (SSTF-BA1401-51) and by the National Research Foundation of Korea (NRF-2019R1A5A1028324).}

\begin{abstract}
We study the product $P_m$ of $m$ real Ginibre matrices with Gaussian elements of size $N$, which has received renewed interest recently. Its eigenvalues, which are either real or come in complex conjugate pairs, become all real with probability one when $m\to\infty$ at fixed $N$. In this regime the statistics becomes deterministic and the Lyapunov spectrum has been derived long ago. On the other hand, when $N\to\infty$ and $m$ is fixed, it can be expected that away from the origin the same local statistics as for a single real Ginibre ensemble at $m=1$ prevails. Inspired by analogous findings for products of complex Ginibre matrices, we introduce a critical scaling regime when the two parameters are proportional, $m=\alpha N$. We derive the expected number, variance and rescaled density of real eigenvalues in this critical regime. This allows us to interpolate between previous recent results in the above mentioned limits when $\alpha\to\infty$ and $\alpha\to0$, respectively.

\end{abstract}

\maketitle

\section{Introduction and discussion of main results} \label{Section_main results}

The study of products of random matrices has been proposed many decades ago by Bellman \cite{MR62368} and by Furstenberg and Kesten \cite{furstenberg1960products}. The motivation was to understand properties of the Lyapunov exponents  \cite{MR826860,MR1145601} in this toy model for chaotic dynamical systems, where usually the matrix dimension $N$ is kept fixed and the number of factors $m$ tends to infinity, see \cite{viana2014lectures} for a recent account. The factors are typically taken to be real non-symmetric or complex non-Hermitian matrices with Gaussian distribution of elements, from the real or complex Ginibre ensemble \cite{ginibre1965statistical}. While Ginibre himself showed the integrability of a single complex ensemble as a determinantal point process, it has taken many joint efforts to put the real Ginibre ensemble on the same footing as a Pfaffian point process, and we refer to \cite{khoruzhenko2011non} and references therein for details. 

More recently it was shown that also products of such random matrices represent determinantal respectively Pfaffian point processes at finite $N$ and $m$, and we refer to \cite{AI15} for a review. Naturally, multiplying real Ginibre matrices has been the most challenging in that respect, cf. \cite{KI14,MR3551633} for results. What is special about multiplying real matrices is that its eigenvalues are either real, or come in complex conjugated  pairs. It was observed numerically early on by Sommers and coworkers \cite{sommers1988spectrum} for a single matrix $m=1$, that the fraction of real eigenvalues is $O(\sqrt{N})$. This fact was proven later by Edelman, Kostlan and Shub \cite{MR1231689}. Much more recently Lakshminarayan \cite{Arul13} studied real products numerically in the context of entanglement of quantum systems and found, that for fixed $N$ all eigenvalues become real when $m\to\infty$ with probability one. This was proven by Forrester \cite{MR3159513} and this 
property holds beyond multiplying Gaussian Ginibre ensembles \cite{HJL15,TRR17,MR3903568}.
Both the expected number of real eigenvalues and its variance were determined very recently by Simm and his coworker \cite{MR3685239,fitzgerald2021fluctuations}, as we will recall below, including the opposite limit with $N\to\infty$ at $m$ fixed. Similar considerations have been made very recently for products of truncated orthogonal matrices \cite{little2021number}.

In the limit $N\to\infty$ at fixed $m$, it was shown for products of complex Ginibre matrices that the singular values \cite{liu2016bulk} and complex eigenvalues \cite{liu2019phase} of the product matrix have the same eigenvalue statistics as a single matrix at $m=1$, apart from the origin where the mean density diverges and new universality classes emerge, cf.  \cite{AI15}. It was therefore natural to ask about the existence of a critical regime, that interpolates between such a random matrix behaviour when $N\to\infty$ at $m$ fixed, and the deterministic Lyapunov spectrum when $m\to\infty$ at $N$ fixed \cite{MR826860,MR1145601,MR3055376,MR3249905}. An interpolating kernel was indeed identified \cite{ABK19,ABK20,liu2018lyapunov} that describes the local statistics in a critical regime, when the two parameters are proportional, $m=\alpha N$ with $\alpha>0$ fixed, as conjectured in 
\cite{MR3262164}. 

In this article we will introduce such a critical regime for the product of real Ginibre matrices, with the goal to describe the expected number of real eigenvalues, its variance and density which are all global quantities. Here, we will rely on previous results at finite $N$ and $m$ \cite{MR3551633}, where the Pfaffian structure of the product of real Ginibre matrices was derived. Our results will allow us to interpolate between previous findings as a function of $\alpha$, as it is described below.

\bigskip

Let $m \equiv m_N \in \mathbb{N}$ be a parameter which possibly depends on $N$. 
We denote by $X_j$ ($j=1,\dots, m$) independent real Ginibre matrices of size $N$.
Namely, the entries of $X_j$ are given by independent Gaussian random variables of mean $0$ and variance $\frac{1}{N}.$
Throughout this paper, we only consider the case $N$ is even. 
(The case $N$ is odd can also be treated in a similar way but it requires a separate analysis, cf. \cite{HJS2008,MR2485724}.) 

Let us define the product matrix
\begin{equation} \label{Pm}
 P_m:=N^{-\frac{m}{2}} X_1 \cdots X_m.
\end{equation}
Because $P_m$ is a real asymmetric matrix it has both real eigenvalues and complex 
eigenvalues that come in complex conjugates pairs.
Here, we shall study only the real eigenvalues of $P_m$ in the critical regime when $m$ is proportional to $N$, i.e.  for fixed $\alpha$ we have
\begin{equation} \label{m alphaN}
m= \alpha N, \qquad \alpha \in (0,\infty)
\end{equation} 

Let $\mathcal{N}_{N}(m)$ be the number of real eigenvalues\footnote{Because we only consider even $N$ it will also be even.}.
For fixed $m$ is it now known that only $O(\sqrt{N})$ eigenvalues are real \cite{MR3685239}. In contrast, for fixed $N$ and large $m$ all eigenvalues becomes real \cite{MR3159513}, as we will recall below.  
In the critical regime \eqref{m alphaN}, the situation becomes somewhat similar to the real elliptic Ginibre ensemble in the weakly asymmetric regime \cite{efetov1997directed}, in the sense that a non-trivial portion of the $N$ eigenvalues is real, cf. \cite{byun2021real,FT20}. 

We introduce the notion
\begin{equation} \label{EN VN def}
E_{N}(m):=\mathbb{E} \, \mathcal{N}_{N}(m),  \qquad V_{N}(m):=\textup{Var}\, \mathcal{N}_{N}(m)
\end{equation}
for the mean and the variance of $\mathcal{N}_N(m)$.
We will study the asymptotic behaviour of these two quantities \eqref{EN VN def} when the dependence of $m$ on $N$ is given by \eqref{m alphaN}.

\subsection{Expected number of real eigenvalues}
 
Let us first recall the known asymptotic behaviour of $E_N(m)$. 

\begin{itemize}
    \item \textbf{Fixed $N$ and $m \to \infty$:} 
    For each fixed $N$, it was shown by Forrester \cite[Prop. 6]{MR3159513} that the probability that all eigenvalues of $P_m$ are real tends to $1$ as $m \to \infty.$ cf. \cite{MR3551633}. 
    As a consequence, it follows
    \begin{equation} \label{EN N fixed}
   E_N(m) = N+o(1), \qquad \textup{as } m \to \infty.
   \end{equation}
    \item \textbf{Fixed $m$ and $N \to \infty$:}
    For each fixed $m,$ it was shown by Simm \cite[Thm 1.1]{MR3685239} that 
    \begin{equation} \label{EN m fixed}
E_N(m) = \Big( \frac{2m}{\pi} N \Big)^{\frac12}\cdot (1+o(1)), 
\qquad \textup{as } N \to \infty. 
\end{equation}
  The special case $m=1$ was previously obtained in the pioneering work \cite{MR1231689} of Edelman, Kostlan, and Shub. 
\end{itemize}
One may expect that such results \eqref{EN N fixed} and \eqref{EN m fixed} hold as long as $m \gg N$ for the former case, and $m \ll N$ for the latter case. 
For the product of complex Ginibre matrices such an extension from the case of fixed $N$ respectively $m$ was shown in \cite{ABK19}.

Let us define the following integral over of the error function $\mbox{erf}(x)=\frac{2x}{\sqrt{\pi}}\int_0^1\exp(-x^2s^2) \, ds$
\begin{equation} \label{R(alpha)}
c(\alpha):=  \int_0^1  \erf \Big( \sqrt{ \frac{\alpha}{8t} } \Big) \,dt=\Big( 1+\frac{\alpha}{4} \Big) \erf\Big( \sqrt{ \frac{\alpha}{8} } \Big)-\frac{\alpha}{4}+\sqrt{ \frac{\alpha}{2\pi} }e^{-\frac{\alpha}{8}},
\end{equation}
where the second expression follows from known integrals, see Section~\ref{Section_mean} for more details. 
Then, we obtain the following result for the expected number of real eigenvalues in the critical regime \eqref{m alphaN}.

\begin{thm} \label{Thm_EN}
For each $\alpha \in (0,\infty)$ fixed, we have the following asymptotic when scaling $m=\alpha N$:
\begin{equation}
\label{Ecrit}
\lim_{N \to \infty} \frac{ E_N(\alpha N) }{ N }=c(\alpha).
\end{equation}
\end{thm}

The limiting curve $c(\alpha)$ and a comparison with the fraction of real eigenvalues obtained numerically from finite size matrices is shown in Fig. \ref{EN-numerics}.

\begin{figure}[h!]
		\begin{center}	
		 \includegraphics[width=0.48\textwidth]{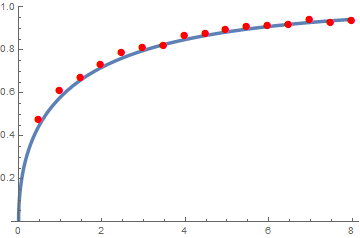}
		\end{center}
    \caption{The plot displays the graph of $E_N(\alpha N)/N$ as a function of $\alpha$ for the example $N=50$ with $50$ samples (dots), and its comparison with $c(\alpha)$ from \eqref{R(alpha)}. Notice that the fluctuations around the limiting curve are still visible for this size of matrices.}
    \label{EN-numerics}
\end{figure}

Recall the asymptotic behaviour of the error function 
\begin{equation} \label{erf asymp 0 inf}
	\erf(x) \sim 
	\begin{cases}
		\frac{2x}{\sqrt{\pi} } &\text{as} \quad x \to 0\,,
		\\
		1 & \text{as} \quad x \to \infty\,,
	\end{cases}
\end{equation}
see e.g. \cite[Eqs.(7.6.1),(7.12.1)]{olver2010nist}. 
Using \eqref{erf asymp 0 inf}, it is easy to observe that 
\begin{equation} \label{R(alpha) asymp 0 inf}
	c(\alpha) \sim 
	\begin{cases}
		\sqrt{ \frac{2\alpha}{\pi} } &\text{as} \quad \alpha \to 0 \, ,
		\\
		1  &\text{as} \quad \alpha \to \infty \, .
	\end{cases}
\end{equation}
In \eqref{R(alpha) asymp 0 inf}, the first limit $\alpha\to0$ matches with \eqref{EN m fixed} when inserting $m=\alpha N$ therein, 
\begin{equation}
\Big( \frac{2 m}{\pi}  N \Big)^{\frac12}=\Big( \frac{2 \alpha N}{\pi}  N \Big)^{\frac12}= \sqrt{ \frac{2\alpha}{\pi} }\,N.
\end{equation}
On the other hand, the second limit $\alpha\to\infty$ in \eqref{R(alpha) asymp 0 inf}  obviously corresponds to \eqref{EN N fixed}. In that sense Theorem \ref{Thm_EN} interpolates between the two previously known cases.

Finally, we also remark the following subleading asymptotic to \eqref{R(alpha) asymp 0 inf}
\begin{equation} \label{c(alpha) fine asym}
 1-c(\alpha) \sim 16\sqrt{ \frac{2}{\pi} }  \frac{ e^{-\frac{\alpha}{8}} }{\alpha\sqrt{\alpha}}, \qquad \textup{as } \alpha \to \infty, 
\end{equation}
see Section~\ref{Section_mean} for more details.

\subsection{Variance of the number of real eigenvalues}

We now discuss the results for the variance $V_N(m)$ in the critical regime \eqref{m alphaN}. 
The asymptotic behaviour of the variance $V_N(m)$ is known for the following cases. 

\begin{itemize}
    \item \textbf{Fixed $N$ and $m \to \infty$:} 
    It again follows from the work of Forrester \cite{MR3159513} that
    \begin{equation} \label{VN N fixed}
   \lim_{m \to \infty} \frac{V_N(m)}{E_N(m)} = 0.
   \end{equation}
    \item \textbf{Fixed $m$ and $N \to \infty$:}
    For each fixed $m,$ it was shown in a recent work by Fitzgerald and Simm \cite[Prop. 3.1]{fitzgerald2021fluctuations} that the following ratio holds independently of $m$:
    \begin{equation} \label{VN m fixed}
 \lim_{N \to \infty}  \frac{V_N(m)}{E_N(m)} = 2-\sqrt{2}.
\end{equation}
  The special case $m=1$ was previously obtained by Forrester and Nagao \cite{forrester2007eigenvalue}. 
\end{itemize}
Again, one would expect that the asymptotic behaviour \eqref{VN N fixed} and \eqref{VN m fixed} remain true as long as $m \gg N$ for the former case, and $m \ll N$ for the latter case.

Let us define
\begin{equation} \label{S(alpha)} 
\begin{split}
s(\alpha)&:=\int_0^1 \frac{\alpha}{8\pi t} \sum_{k=-\infty}^\infty \Big(  \int_{2k-1}^{2k+1} e^{-\frac{\alpha}{8t}x^2 }\,dx \Big)^2 \,dt
\\
&= \int_0^1 \sum_{k=-\infty}^\infty   \Pr( 2k-1 \le X_{t,\alpha} \le 2k+1 ) ^2 \,dt.
\end{split}
\end{equation}
Here $X_{t,\alpha}$ is the normal distribution 
\begin{equation}
X_{t,\alpha}:=\mathcal{N}(0,\sigma^2), \qquad \sigma:=2\sqrt{\frac{t}{\alpha}}.
\end{equation}

From the expression \eqref{S(alpha)}, one can observe that $s(\alpha)$ is an increasing function. 
Furthermore, it is easy to see that $s(\alpha)$ is bounded above by $1$ since  
\begin{equation}
s(\alpha) \le \int_0^1  \Big(  \sum_{k=-\infty}^\infty   \Pr( 2k-1 \le X_{t,\alpha} \le 2k+1 )  \Big)^2 \,dt =\int_0^1 \,1\,dt=1. 
\end{equation}
Hence the integral and sum interchange.
See Figure~\ref{Fig_S(alpha)} for the graph of $s(\alpha).$

\begin{figure}[h!]
		\begin{center}	
		 \includegraphics[width=0.48\textwidth]{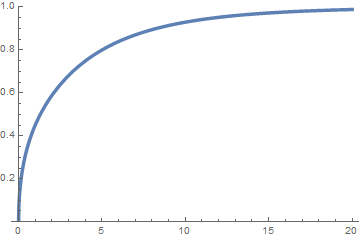}
		\end{center}
    \caption{The graph of $s(\alpha)$ from \eqref{S(alpha)} is shown as a function of $\alpha$.} \label{Fig_S(alpha)}
\end{figure} 

We obtain the following asymptotic for the variance normalised by the mean in the critical regime \eqref{m alphaN}.

\begin{thm}\label{Thm_VN}
For each $\alpha \in (0,\infty)$ fixed, we have in the critical scaling regime $m=\alpha N$
\begin{equation}
\label{r-def}
\lim_{N \to \infty} \frac{ V_N(\alpha N) }{ E_N(\alpha N) }=r(\alpha):=2-2 \, \frac{s(\alpha)}{c(\alpha)}.
\end{equation}
\end{thm}

In Proposition~\ref{Prop_s(alpha) zero}, we show the following asymptotic for \eqref{S(alpha)}
\begin{equation} \label{S(alpha) inf 0}
s(\alpha) \sim 
	\begin{cases}
	  \sqrt{ \frac{ \alpha }{\pi} } &\text{as} \quad \alpha \to 0 \, ,
		\\
		1  &\text{as} \quad \alpha \to \infty \, .
	\end{cases}
\end{equation} 
Combining \eqref{R(alpha) asymp 0 inf} and \eqref{S(alpha) inf 0}, we thus have asymptotically for Theorem \ref{Thm_VN}
\begin{equation} \label{r(alpha) asymp 0 inf}
	r(\alpha) \sim 
	\begin{cases}
	  2-\sqrt{2} &\text{as} \quad \alpha \to 0 \, ,
		\\
		0  &\text{as} \quad \alpha \to \infty \, .
	\end{cases}
\end{equation}
One can notice that the asymptotic behaviour \eqref{r(alpha) asymp 0 inf} interpolates between \eqref{VN m fixed} and \eqref{VN N fixed} at fixed $m$ and $N$, respectively. 
The graph of $r(\alpha)$ is shown in Fig. \ref{VN-numerics} and compared with numerical simulations.

\begin{figure}[h!]
\begin{subfigure}{0.48\textwidth}
		\begin{center}	
		 \includegraphics[width=\textwidth]{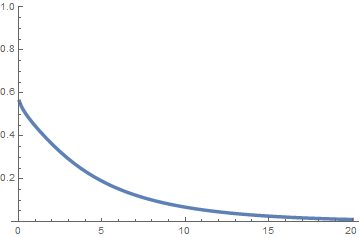}
		\end{center}
		\subcaption{$r(\alpha)$}
	\end{subfigure}	
	\begin{subfigure}{0.48\textwidth}
		\begin{center}	
		 \includegraphics[width=\textwidth]{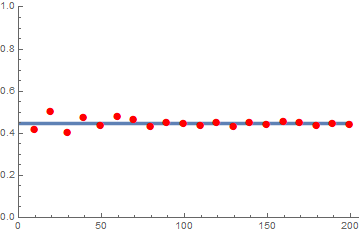}
		\end{center}
		\subcaption{$r(1)$}
	\end{subfigure}	
    \caption{Plot (A) shows the graph of the function $r(\alpha)$ from Theorem \ref{Thm_VN}. In figure (B) the variance of real eigenvalues of $P_m$ divided by its expected number is determined for $N=m=50$ at the specific value $\alpha=1$, with $L$ samples. It is compared with $r(\alpha=1)\approx 0.45$ (full line), where $L$ takes values $L=10,20,\dots,200$ (red dots).}
    \label{VN-numerics}
\end{figure}

As an immediate consequence of Theorems~\ref{Thm_EN} and ~\ref{Thm_VN}, we obtain the following convergence of the random variable $\mathcal{N}_N(\alpha N)/N$ in probability. 

\begin{cor}
For each $\alpha \in (0,\infty)$ fixed, we have in the critical regime $m=\alpha N$
\begin{equation}
 \frac{ \mathcal{N}_N(\alpha N) }{ N } \to c(\alpha)
\end{equation}
as $N \to \infty$, in probability. 
\end{cor}

\subsection{Densities of real eigenvalues}

We denote by $\bfR_{N,1}^m$ the $1$-point function of real eigenvalues \cite{MR3551633}, see Section \ref{Section_Prelim} for its definition in term of the kernel of the underlying Pfaffian point process. 
Let us write 
\begin{equation} \label{rhoN m bfRN1}
\rho_N^m(x):=\frac{1}{E_N(m)} \bfR_{N,1}^m(x)
\end{equation}
for the normalised density of real eigenvalues of $P_m$ in \eqref{Pm}. 

We also consider the rescaling of its argument $x$ 
\begin{equation} \label{Lyapunov resacling}
 \lambda=\begin{cases}
  x^{ \frac{1}{m} } &\textup{if }x>0,
  \\
  - |x|^{ \frac{1}{m} } &\textup{if }x<0.
 \end{cases}
\end{equation}
Such a rescaling is well known in the study of Lyapunov \cite{MR826860,MR3249905,MR3055376} and stability exponents \cite{MR3903568,MR3262164,MR3335710}.
By the change of variable \eqref{Lyapunov resacling}, the associated density $ \wt{\rho}_N^m(\lambda)$ is given by 
\begin{equation} \label{rho wt rho}
\wt{\rho}_N^m(\lambda)= m\,\lambda^{m-1}\,\rho_N^m(\lambda^m),
\end{equation}
where $\rho_N^m$ is given by \eqref{rhoN m bfRN1}.
Here and in the sequel, we shall use the tilde notation for the quantities associated with the scaled variable $\lambda$. Let us summarise again what can be said about the two densities in the limits with $N$ respectively $m$ fixed.

\begin{itemize}
    \item \textbf{Fixed $N$ and $m \to \infty$:} 
    For each fixed $N$, one can expect that
    \begin{equation} \label{density N fixed}
    \lim_{m \to \infty} \rho_{N}^m(x) = \delta(x)
    \end{equation}
    in the sense of distribution, see the remark below. On the other hand, it follows from Newman's celebrated work \cite{MR826860}, and the equivalence between Lyapunov and stability exponents (see e.g. \cite{MR3903568}) that 
    \begin{equation} \label{rho wt lim triangle}
    \lim_{m \to \infty} \wt{\rho}_N^m(\lambda)=  \mathbbm{1}_{(-1,1)} (\lambda) \cdot |\lambda|. 
    \end{equation}
    The density in \eqref{rho wt lim triangle} is known as Newman's triangular law. 
    \smallskip 
    \item \textbf{Fixed $m$ and $N \to \infty$:}
   It was shown in \cite[Thm. 1.2]{MR3685239} that for a fixed $m$, the weak convergence
 \begin{equation} \label{density m fixed}
\lim_{N \to \infty} \rho_{N}^m(x)= \rho^m(x):= \mathbbm{1}_{(-1,1)} (x) \cdot \frac{1}{2m} |x|^{ \frac{1}{m}-1 }
\end{equation}
holds, see also a recent work \cite{fitzgerald2021fluctuations}.
The convergence \eqref{density m fixed} was previously conjectured in \cite{MR3551633}.
Equivalently, by \eqref{rho wt rho} and \eqref{density m fixed}, for any fixed $m \in \mathbb{N}$, 
\begin{equation} \label{rho wt lim circ}
\lim_{N \to \infty}\wt{\rho}_N^m(\lambda)=  \wt{\rho}^m(\lambda):=  \mathbbm{1}_{(-1,1)} (\lambda) \cdot \frac12. 
\end{equation}
Note that the limiting density $\wt{\rho}^m$ in \eqref{rho wt lim circ} recovers the constant density of real eigenvalues, previously known from \cite{MR1231689} for $m=1$.
\end{itemize}

As before we now define an interpolating density for $\alpha>0$ as follows
\begin{equation} \label{rho wt}
    \wt{\rho}_\alpha(\lambda):= \mathbbm{1}_{(-1,1)} (\lambda) \cdot \frac{1}{c(\alpha)}\, |\lambda| \erf\Big( \sqrt{ \frac{\alpha}{8} } \frac{1}{|\lambda|}  \Big).
\end{equation}
It appears in the following convergence result.

\begin{thm} \label{Thm_density}
For each $\alpha \in (0,\infty)$ fixed, the following results hold in the critical regime \eqref{m alphaN}.
\begin{enumerate}[label=(\roman*)]
    \item The density $\rho_N^{\alpha N}$ converges weakly to the Dirac delta measure as $N \to \infty$ in the sense of distribution, i.e. for every bounded and continuous function $f$, we have 
\begin{equation}
\int_\R f(x) \rho_N^{\alpha N}(x)\,dx \to f(0), \qquad \mbox{as}\quad N \to \infty.
\end{equation}
    \item 
The rescaled density $\wt{\rho}_N^{\alpha N}$ converges weakly to $\wt{\rho}_\alpha$ as $N \to \infty$ in the sense of distribution, i.e. for every bounded and continuous function $f$, we have 
\begin{equation}
\int_\R f(\lambda) \wt{\rho}_N^{\alpha N}(\lambda)\,d\lambda \to \int_\R f(\lambda) \wt{\rho}_{\alpha}(\lambda)\,d\lambda,  \qquad\mbox{as}\quad  N \to \infty.
\end{equation}
\end{enumerate}

\end{thm}

In Figure \ref{rhotilde} we illustrate the limiting rescaled density $\wt{\rho}_\alpha$ and compare it with numerical simulations.

\begin{figure}[h!]
\begin{subfigure}{0.48\textwidth}
		\begin{center}	
		 \includegraphics[width=\textwidth]{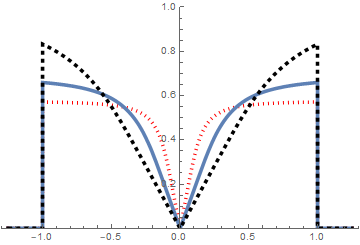}
		\end{center}
		\subcaption{$\wt{\rho}_\alpha$}
	\end{subfigure}	
	\begin{subfigure}{0.48\textwidth}
		\begin{center}	
		 \includegraphics[width=\textwidth]{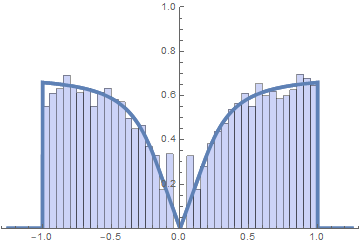}
		\end{center}
		\subcaption{$\wt{\rho}_1$}
	\end{subfigure}	
	\caption{ Plot (A) shows the graph of $\wt{\rho}_\alpha(x)$ in \eqref{rho wt} as a function of $x$ for a few values of $\alpha$: $\alpha=0.2$ (dotted, red line), $\alpha=1$ (full, blue line), and $\alpha=5$ (dashed, black line). While the density  vanishes at the origin, $\wt{\rho}_\alpha(0)=0$ for all $\alpha,$ the interpolation between an almost constant density for small $\alpha$ and the linear density for large $\alpha$ is clearly visible.
Plot (B) displays the graph of $\wt{\rho}_{\alpha=1}(x)$ for $\alpha=1$ and its comparison with the histogram of real eigenvalues obtained numerically. Here, we use matrices at $N=m=50$ with $L=200$ samples.}\label{rhotilde} 
\end{figure}

\begin{rmk*}
The  degenerate limit to the Dirac delta in Theorem \ref{Thm_density} (i) and in \eqref{density N fixed} can be expected from \eqref{density m fixed}.
To be more precise, notice that for each $k \in \mathbb{Z}_{\ge 0}$,
\begin{equation}
\int_{-1}^1 x^k\,\rho^m(x)\,dx = \frac{1+(-1)^k}{2(1+km)} \to \begin{cases}
1 & k=0,
\\
0 & k \ge 1,
\end{cases} \qquad (m \to \infty). 
\end{equation}
This gives that for a test function $f$,
\begin{equation}
\int_{-1}^1 f(x)\,\rho^m(x)\,dx \to f(0), \qquad (m \to \infty).
\end{equation}
\end{rmk*}

Note that by \eqref{erf asymp 0 inf} and \eqref{R(alpha) asymp 0 inf}, we have 
\begin{equation} \label{rho(alpha) asymp 0 inf}
	\wt{\rho}_\alpha(\lambda) \sim 
	\begin{cases}
	  1/2 &\text{as} \quad \alpha \to 0 \, ,
		\\
		|\lambda|  &\text{as} \quad \alpha \to \infty \, ,
	\end{cases} \qquad \lambda \in (-1,1).
\end{equation}
The limit $\alpha \to 0$ in \eqref{rho(alpha) asymp 0 inf} corresponds to a constant density on the real axis \cite{MR1231689}, independently of $m$, whereas the opposite limit $\alpha \to \infty$ in \eqref{rho(alpha) asymp 0 inf} corresponds to Newman's triangular law \cite{MR826860,MR1145601}. In that sense the density $\wt{\rho}_\alpha$ again interpolates between the cases \eqref{rho wt lim triangle}  and \eqref{rho wt lim circ} at fixed $m$ and $N$, respectively.

\medskip 

\noindent {\bf Organisation of the paper.} The remainder of this paper is organised as follows. 
In Section~\ref{Section_Prelim}, we summarise the
relevant material on the integrable structure of the product matrix $P_m$, following 
closely \cite{MR3159513,MR3551633}. 
Furthermore, we derive a certain asymptotic behaviour of the Meijer $G$-function (Proposition~\ref{Prop_ajk asymp}) that will be crucial in the following sections. 
Section~\ref{Section_mean} is devoted to the study of the expected number of real eigenvalues and to the proof of Theorem~\ref{Thm_EN}. Furthermore, we derive properties of the function $c(\alpha)$.
In Section~\ref{Section_variance}, we address the variance of the number of real eigenvalues and prove Theorem~\ref{Thm_VN}.  
In Proposition~\ref{Prop_s(alpha) zero}, we present the asymptotic behaviour of the function $s(\alpha)$.  
Finally in Section~\ref{Section_density}, we show the results for the densities of real eigenvalues and prove Theorem~\ref{Thm_density}.

\section{Preliminaries}\label{Section_Prelim}

In this section, we compile some known facts about the real eigenvalues of the matrix $P_m$ in \eqref{Pm} where we follow \cite{MR3159513,MR3551633}.  
In particular we show a certain asymptotic behaviour of the Meijer $G$-function (Proposition~\ref{Prop_ajk asymp}) which plays an important role in the proofs of our main results in the following sections. 

\bigskip 

The Meijer $G$-function is defined by the Mellin-Barnes integral representation 
\begin{equation} \label{Meijer G}
\MeijerG {m,n} {p,q} { a_1,\dots,a_p\\
   b_1,\dots,b_q| } {z}:= \frac{1}{2\pi i} \oint_L  z^s \frac{     \prod_{j=1}^{m}\Gamma(b_j-s)   \prod_{j=1}^{n}\Gamma(1-a_j+s)      }{  \prod_{j=1+m}^{q}\Gamma(1-b_j+s) \prod_{j=1+n}^{p}\Gamma(a_j-s)       }\,ds,
\end{equation}
where the integration contour $L$ in $\C$ depends on the poles of the gamma functions. To be more precise, $L$ is properly chosen so that it separates the poles of the factors $\Gamma(b_j-s)$ from those of the factors $\Gamma(1-a_j+s)$, see e.g. \cite[Chapter 16]{olver2010nist}.

Let us write
\begin{equation} \label{omega}
\omega(x):=\MeijerG{m,0}{0,m}{-\\0,\dots,0}{\frac{x^2}{2^m}}
\end{equation}
for the weight function of real eigenvalues. 
The eigenvalues of $P_m$ form a Pfaffian point processe with the weight $\omega$ following 
closely \cite{MR3159513}, where one has to distinguish between correlations among complex eigenvalues, real eigenvalues or mixed correlations. In the following we will only recall the results about the correlation fucntions among real eigenvalues. 
It was shown in \cite{MR3551633} that the real eigenvalues of the product matrix $P_m=X_1 \cdots X_m$ represent a Pfaffian point process with  correlation kernel 
\begin{equation}
K_N(x,y)= \begin{pmatrix}
D_N(x,y) & S_N(x,y) 
\smallskip 
\\
-S_N(y,x) & \tilde{I}_N(x,y) 
\end{pmatrix},
\end{equation}
where 
\begin{equation} \label{SN}
S_N(x,y)=\sum_{j=0}^{N-2} \frac{\omega(x)\,x^j}{ (2\sqrt{2\pi} j!)^m } (x A_j(y)-A_{j+1}(y))
\end{equation}
is the main building block, sometimes called scalar kernel, together with\footnote{Notice that in \cite[Eq. (4.8)]{MR3551633} there is a typo as the sign in front of the integral in $\tilde{I}_N$ should be positive. This follows from its definition in \cite[Sect. 4.2]{MR3551633}, as confirmed by the authors.} 
\begin{equation}
D_N(x,y):= -\frac{\pa}{\pa y} S_N(x,y) , \qquad 
	\tilde{I}_N(x,y):= \int_{x}^{y} S_N(t,y)\,dt +\frac12 \sgn(x-y).
\end{equation}
Furthermore, it holds that  \cite{MR3551633}
\begin{equation} \label{Aj}
\begin{split}
A_j(y)=\int_\R \omega(v)\sgn(y-v)v^j\,dv
=  \begin{cases}
-2^{\frac{m(j+1)}{2}} \MeijerG{m+1,0}{1,m+1}{1\\ 0, \frac{j+1}{2}, \dots, \frac{j+1}{2} }{ \dfrac{y^2}{2^m} } & j \textup{ odd,}
\smallskip 
\\
y^{j+1}  \MeijerG{m,1}{1,m+1}{ -\frac{j-1}{2} \\ 0, \dots,0, -\frac{j+1}{2} }{ \dfrac{y^2}{2^m} } & j \textup{ even.}
\end{cases} 
\end{split}
\end{equation}

Following \cite{MR3159513}, let us write 
\begin{equation} \label{ajk Meijer G}
a_{j,k}:= \MeijerG{m+1,m}{m+1,m+1}{\frac32-j,\dots,\frac32-j,1\\
0, k, \dots, k |}{1}.
\end{equation}
It is also convenient to introduce 
\begin{equation} \label{bjk ajk}
b_{j,k}:=\frac{ a_{ j , k } }{ ( \Gamma(j-\frac12)\Gamma(k) )^m }.
\end{equation}

The $k$-point correlation functions of real eigenvalues $\bfR_{N,k}^m$ can thus be written as 
\begin{equation}
\label{RNkdef}
\bfR_{N,k}^m(x_1,\ldots,x_k):=\mbox{Pf} [ K_N(x_j,x_l) ]_{j,l=1}^k.
\end{equation}
In particular we obtain for the 1- and 2-point function needed later
\begin{align}
\label{RN1}
\bfR_{N,1}^m(x)&=S_N(x,x),
\\
\bfR_{N,2}^m(x,y)&=S_N(x,x)S_N(y,y)-D_N(x,y)\tilde{I}_N(x,y)+S_N(x,y)S_N(y,x).
\label{RN2}
\end{align}
Notice that these $k$-point functions are not normalised, in particular because the number of real eigenvalues is a random variable itself. Furthermore, we observe that due to the symmetry of $A_j$ \eqref{Aj} together with \eqref{SN}, the 1-point function is an even function in $x$.

Next, we collect moments of the weight function $\omega$ in \eqref{omega} based on the building blocks \eqref{Aj}. 
The following property was shown in \cite{MR3159513,MR3551633}.

\begin{lem} \label{Lem_int of omega}
The weight function $\omega$ in \eqref{omega} together with the functions $A_j$ in \eqref{Aj} satisfy
\begin{equation} \label{omega A int ajk}
\int_{\R^2} \omega(x)\, \omega(y)\, x^{j-1} \, y^{k-1} \sgn(y-x)\,dx\,dy=
-\int_\R \omega(x) \,x^{j-1} \,A_{k-1}(x) \,dx 
= \alpha_{j,k}.
\end{equation}
Due to the symmetry of the weight function the coefficients have the following properties
\begin{equation} \label{alpha-sym}
\alpha_{j,k}=-\alpha_{k,j}, \quad \mbox{and}\quad
\alpha_{j,k}=0 \quad \mbox{if}\ \  j+k\ \ \mbox{even}. 
\end{equation}
They can be expressed as  
\begin{equation} \label{alpha a jk}
\alpha_{2j-1,2k}=2^{ (j+k-\frac12)m } a_{j,k}, 
\end{equation}
where $a_{j,k}$ is given by \eqref{ajk Meijer G}. 
\end{lem}

Let us recall the following lemma shown in \cite[Proposition 6]{MR3159513}. 
By the expression \eqref{EN sum} of $E_N$ below, this immediately gives \eqref{EN N fixed}. 

\begin{lem} \label{Lem_ajk asymp Forrester} \textup{(Cf. \cite[Proposition 6]{MR3159513})}
For each fixed $j,k$, we have
\begin{equation} \label{bjk asymp Forrester}
\lim_{m\to\infty}  b_{j,k} = \begin{cases}
1 & j \le k ,
\\
0 & j> k.
\end{cases}
\end{equation}
\end{lem}

We emphasise that Lemma~\ref{Lem_ajk asymp Forrester} will not be used further in our proofs below. 
Nevertheless, it is instructive to compare this lemma with the following proposition. 

\begin{prop} \label{Prop_ajk asymp}
Let $j=tN/2$ with $t \in (0,1)$ and $m=\alpha N$. Let $l=k-j$ be fixed.
Then as $N \to \infty$, we have 
\begin{equation} \label{bjk asymp}
b_{j,k} = \frac12 \Big[ 1+\erf \Big( (2l+1)\sqrt{ \frac{\alpha}{8t} } \Big)  \Big]\cdot (1+o(1)).
\end{equation}
\end{prop}

Notice that for $\alpha \gg 1$, the asymptotic behaviour of \eqref{bjk asymp} recovers \eqref{bjk asymp Forrester}.

\begin{proof}[Proof of Proposition~\ref{Prop_ajk asymp}]

By \eqref{Meijer G} and \eqref{ajk Meijer G}, for $j \ge 1$, we have 
\begin{equation}
a_{j,k}= \frac{1}{2\pi i} \int_C ( \Gamma(j-\tfrac12+s) \Gamma(k-s) )^m \, \frac{ \Gamma(-s) }{ \Gamma(1-s) }\,ds
=-\frac{1}{2\pi i} \int_C \frac{ ( \Gamma(j-\frac12+s) \Gamma(k-s) )^m  }{s}\,ds,
\end{equation}
where $C$ can be taken to be a contour starting at $-i\infty$, passing through the real axis within the interval $(\frac12-j,0)$, and finishing at $i \infty$.
Then by \eqref{bjk ajk}, we obtain the contour integral representation
\begin{equation} \label{ajk Gamma contour}
b_{j,k} =-\frac{1}{2\pi i} \int_C \Big(  \frac{\Gamma(j-\frac12+s) }{ \Gamma(j-\frac12)  }  \frac{ \Gamma(k-s)}{ \Gamma(k) }  \Big)^m\,\frac{ds}{s}. 
\end{equation}

Recall here that the gamma function satisfies the asymptotic behaviour
\begin{equation} \label{Gamma ratio general}
\frac{\Gamma(z+a)}{\Gamma(z+b)} = z^{a-b} \Big( 1+\frac{(a-b)(a+b-1)}{2z}+O(z^{-2}) \Big), \qquad (z \to \infty),
\end{equation}
see e.g. \cite[Eq.(5.11.13)]{olver2010nist}.
Using this, for $j=t N/2$ with $t \in (0,1)$ and fixed $l=k-j$, we have
\begin{equation}
\frac{\Gamma(j-\frac12+s) }{ \Gamma(j-\frac12)  } = j^s \Big( 1+\frac{s^2-2s}{2j}+O(j^{-2}) \Big)
\end{equation}
and
\begin{equation}
\frac{\Gamma(k-s)}{ \Gamma(k) }= \frac{\Gamma(j+l-s)}{ \Gamma(j+l) } = j^{-s} \Big( 1+\frac{s^2+(1-2l)s}{2j}+O(j^{-2}) \Big).
\end{equation}
Thus we have 
\begin{equation}
\frac{\Gamma(j-\frac12+s) }{ \Gamma(j-\frac12)  } \frac{\Gamma(k-s)}{ \Gamma(k) }= 1+\frac{ 2s^2-(1+2l)s }{ 2j }+O(j^{-2}).
\end{equation}
Since $m=\alpha N$, this gives
\begin{equation} \label{Gamma ratio asym}
\Big(  \frac{\Gamma(j-\frac12+s) }{ \Gamma(j-\frac12)  }  \frac{ \Gamma(k-s)}{ \Gamma(k) }  \Big)^m \sim  e^{ \frac{\alpha}{t} s(2s-1-2l)}.
\end{equation}

Let us write $C_-$ (resp., $C_+$) for a contour that runs from $-i\infty$ to $i \infty$ and passes to the left (resp., right) of the origin.
Then by \eqref{ajk Gamma contour} and \eqref{Gamma ratio asym}, we have 
\begin{equation} \label{bjk contour integral}
b_{j,k}  =  -\frac{1}{2\pi i} \int_{C_-}  e^{ \frac{\alpha}{t} s(2s-1-2l)} \, \frac{ ds }{ s } \cdot (1+o(1)) .
\end{equation}
Therefore it remains to evaluate the contour integral \eqref{bjk contour integral}. 
For this purpose, we shall apply the inverse Laplace transform, 
\begin{equation} \label{Laplace transform}
\frac{1}{2\pi i} \int_{C_+}  e^{ y\,v }  \, \frac{dv}{\sqrt{v}\,(\sqrt{v}+a)}= e^{a^2y} \erfc(a\sqrt{y}), \qquad (a<0)
\end{equation}
that can be found for instance in \cite[17.13.101]{Gradshteyn}.
Notice that for $l \ge 0$,
\begin{equation}
\begin{split}
&\quad -\frac{1}{2\pi i} \int_{C_-}  e^{ x s(2s-1-2l)} \, \frac{ ds }{ s }   = \frac{1}{2\pi i} \int_{C_+}  e^{ x s(2s+1+2l)} \, \frac{ ds }{ s }  
=  \frac12 \frac{1}{2\pi i} \int_{C_+}  e^{ x\,\frac{v-(2l+1)^2}{8}} \, \frac{dv}{\sqrt{v}\,(\sqrt{v}-(2l+1))},
\end{split}
\end{equation}
where we have substituted $v=8(2s^2+s(1+2l))+(1+2l)^2$.
Then by \eqref{Laplace transform}, we obtain 
\begin{equation} \label{contour int l pos}
-\frac{1}{2\pi i} \int_{C_-}  e^{ x s(2s-1-2l)} \, \frac{ ds }{ s } =  \frac12 \erfc\Big(-(2l+1)\sqrt{\frac{x}{8}}\Big)=\frac12\Big[1+\erf\Big((2l+1)\sqrt{\frac{x}{8}}\Big)\Big].
\end{equation}
This gives the asymptotic \eqref{bjk asymp} for $l \ge 0$.

On the other hand, by the Cauchy's formula, we have 
\begin{equation}
\begin{split}
&\quad -\int_{C_-}  e^{ x s(2s-1-2l) } \, \frac{ ds }{ s } -\int_{C_-}  e^{ x s(2s+1+2l) } \, \frac{ ds }{ s } 
\\
&= -\int_{C_-}  e^{ x s(2s-1-2l) } \, \frac{ ds }{ s } +\int_{C_+}  e^{ x s(2s-1-2l) } \, \frac{ ds }{ s } 
= \oint_{C_0}  e^{ x s(2s-1-2l) } \, \frac{ ds }{ s } =2\pi i. 
\end{split}
\end{equation}
Here, $C_0$ denotes a contour encircling the origin counterclockwise.
By \eqref{contour int l pos}, we have that for $l < 0$,
\begin{equation}
\begin{split}
-\frac{1}{2\pi i} \int_{C_-}  e^{ x s(2s-1-2l) } \, \frac{ ds }{ s }& =1+\frac{1}{2\pi i} \int_{C_-}  e^{ x s(2s+1+2l) } \, \frac{ ds }{ s }  
\\
&=  1-\frac12\Big[1+\erf\Big((-2l-1)\sqrt{\frac{x}{8}}\Big)\Big]=\frac12\Big[1+\erf\Big((2l+1)\sqrt{\frac{x}{8}}\Big)\Big],
\end{split}
\end{equation}
which gives the asymptotic \eqref{bjk asymp} for $l < 0$.
Now the proof is complete.
\end{proof}

\section{Proof of Theorem \ref{Thm_EN} and asymptotic of the expected number of real eigenvalues}\label{Section_mean}

In this section, we prove Theorem~\ref{Thm_EN}. 
The key ingredients of the proof are the analytic expression of $E_N$ given in Lemma~\ref{Lem_EN expression} and the asymptotic behaviour of $b_{j,k}$ given in Proposition~\ref{Prop_ajk asymp}.

\bigskip 

The following expression of $E_N$ was obtained by Forrester and Ipsen \cite{MR3551633}. 
Nevertheless, let us recall the proof since this will be instructive to perform similar but more complicated computations in Lemma~\ref{Lem_SN integral}.

\begin{lem} \textup{(Cf. \cite[Subsection 4.2]{MR3551633})} \label{Lem_EN expression}
For each $N$ and $m$, we have 
\begin{equation} \label{EN sum}
E_N(m)= 2 \sum_{j=0}^{N/2-1}   b_{ j+1 , j+1 }  - 2 \sum_{j=0}^{N/2-2}   b_{ j+2 , j+1 } .
\end{equation}
\end{lem}
\begin{proof}
It follows from the definition of the $1$-point function $\bfR_{N,1}^m$ that 
\begin{equation}
E_N(m)= \int_\R  \bfR_{N,1}^m(x)\,dx= \int_\R S_N(x,x)\,dx.
\end{equation}
By \eqref{omega A int ajk}, we have 
\begin{equation}
\int_\R \omega(x)\,x^{j+1} A_j(x) \,dx=-\alpha_{j+2,j+1}, \qquad \int_\R \omega(x)\,x^{j} A_{j+1}(x) \,dx=-\alpha_{j+1,j+2}.
\end{equation}
Then by \eqref{SN}, we have 
\begin{equation}
\begin{split}
\int_\R S_N(x,x)\,dx & =\sum_{j=0}^{N-2} \int_\R  \frac{\omega(x)\,x^j}{ (2\sqrt{2\pi} j!)^m } (x A_j(x)-A_{j+1}(x))\,dx 
\\
&= -\sum_{j=0}^{N-2}   \frac{  \alpha_{j+2,j+1}-\alpha_{j+1,j+2}  }{ (2\sqrt{2\pi} j!)^m }=-2\sum_{j=0}^{N-2}   \frac{  \alpha_{j+2,j+1}  }{ (2\sqrt{2\pi} j!)^m }
\\
&=2\sum_{j=0}^{N/2-1}   \frac{  \alpha_{2j+1,2j+2}  }{ (2\sqrt{2\pi} (2j)!)^m }-2\sum_{j=0}^{N/2-2}   \frac{  \alpha_{2j+3,2j+2}  }{ (2\sqrt{2\pi} (2j+1)!)^m }.
\end{split}
\end{equation}
Now it follows from the duplication formula
\begin{equation}\label{Gamma duplication}
\Gamma(2z+1)=\frac{2^{2z} }{ \sqrt{\pi} } 
\Gamma(z+1)\Gamma(z+\tfrac12), 
\end{equation}
that we have the following identities to used frequently later:
\begin{equation}
\label{Gamma-id}
\frac{2^{2j}}{\sqrt{\pi}(2j)!}=\frac{1}{\Gamma(j+\frac12)\Gamma(j+1)},\qquad
\frac{2^{2j+1}}{\sqrt{\pi}(2j+1)!}=\frac{1}{\Gamma(j+\frac32)\Gamma(j+1)}.
\end{equation}
Together with \eqref{alpha-sym} this implies that
\begin{equation} 
\begin{split}
E_N(m) 
&= 2 \sum_{j=0}^{N/2-1}   \frac{ a_{ j+1 , j+1 } }{ ( \Gamma(j+\frac12)\Gamma(j+1) )^m }  - 2 \sum_{j=0}^{N/2-2}   \frac{  a_{ j+2 , j+1 } }{ ( \Gamma(j+\frac32) \Gamma(j+1) )^m } .
\end{split}
\end{equation}
This completes the proof.
\end{proof}

Using Proposition~\ref{Prop_ajk asymp} and Lemma~\ref{Lem_EN expression}, we prove Theorem~\ref{Thm_EN}.

\begin{proof}[Proof of Theorem~\ref{Thm_EN}]
By \eqref{EN sum} and Proposition~\ref{Prop_ajk asymp}, the Riemann sum approximation gives rise to 
\begin{equation}
\begin{split}
\frac{ E_N(\alpha N) }{N} &= \frac{2}{N} \sum_{j=0}^{N/2-1}   \frac{ a_{ j+1 , j+1 } }{ ( \Gamma(j+\frac12)\Gamma(j+1) )^m }  - \frac{2}{N} \sum_{j=0}^{N/2-2}   \frac{  a_{ j+2 , j+1 } }{ ( \Gamma(j+\frac32) \Gamma(j+1) )^m }   
\\
&\quad \sim \frac12 \int_0^1 1+\erf \Big( \sqrt{ \frac{\alpha}{8t} } \Big) \,dt  - \frac12 \int_0^1 1-\erf \Big( \sqrt{ \frac{\alpha}{8t} } \Big) \,dt = c(\alpha),
\end{split}
\end{equation}
where $c(\alpha)$ is given by \eqref{R(alpha)}.
This completes the proof. 
\end{proof}

Let us briefly indicate how the integral in \eqref{R(alpha)} can be performed in terms of elementary functions, as given in the second equation therein. Setting $y=\sqrt{\alpha/8}$ and inserting the integral representation of the error function, we obtain
\begin{equation}
c(\alpha=8y^2)=\frac{2}{\sqrt{\pi}} \int_0^1 \int_0^1 \frac{y}{\sqrt{t}}e^{-y^2s^2/t}ds\,dt.
\end{equation}
Performing the $t$-integral first and substituting $z=1/t$ we obtain
\begin{align}
\int_0^1 \frac{y}{\sqrt{t}}e^{-y^2s^2/t}\, dt &= y\int_1^\infty e^{-y^2s^2z}z^{-\frac32} \, dz
= 2y\,e^{-y^2s^2}-2y^3s^2\int_0^\infty e^{-y^2s^2(u+1)}\frac{du}{\sqrt{u+1}}
\\
&= 2y\,e^{-y^2s^2}-2s\sqrt{\pi}y^2(1-\erf(ys)), 
\end{align}
where in the second step we have used \cite[3.362.2]{Gradshteyn}. This leads to 
\begin{equation}\label{cy}
c(8y^2)=2\erf(y)-2y^2+4y^2\int_0^1s\erf(ys)ds.
\end{equation}
The remaining integral can be determining by using an integration by parts for $\erf(ys)$ from \cite[5.41]{Gradshteyn}, leading to 
\begin{equation}
\int_0^1 s\erf(ys)ds =\erf(y)+\frac{1}{y\sqrt{\pi}}e^{-y^2} -\int_0^1 s\erf(ys)ds -\frac{1}{2y^2}\erf(y).
\end{equation}
Inserting the resulting expression for the integral into \eqref{cy}, 
\begin{equation}
c(8y^2)=2\erf(y)-2y^2+4y^2\Big(\frac12 \erf(y)  +\frac{1}{2y\sqrt{\pi}}e^{-y^2}-\frac{1}{4y^2}\erf(y)\Big) 
\end{equation}
leads to the desired result for  \eqref{R(alpha)}.

We end this section by showing the asymptotic behaviour \eqref{c(alpha) fine asym}.
Recall that the error function satisfies the expansion 
\begin{equation}
\erf(x)\sim 1-\frac{ e^{-x^2} }{ \sqrt{\pi} x } \sum_{m=0}^\infty (-1)^m \frac{(2m-1)!!}{(2x^2)^m} , \qquad (x \to \infty),
\end{equation}
see e.g. \cite[Eq.(7.12.1)]{olver2010nist}.
Using this, we have
\begin{equation}
 \erf\Big( \sqrt{ \frac{\alpha}{8} } \Big) \sim 1-\sqrt{ \frac{8}{\pi \alpha} } e^{-\frac{\alpha}{8}} \sum_{m=0}^\infty (-1)^m \frac{ (2m-1)!! }{ (\alpha/4)^m }, \qquad (\alpha \to \infty).
\end{equation}
This gives rise to
\begin{align*}
c(\alpha)&= \Big( 1+\frac{\alpha}{4} \Big) \Big[ 1-e^{-\frac{\alpha}{8}} \sqrt{\frac{8}{\pi \alpha}} \Big( 1-\frac{4}{\alpha}+\frac{48}{\alpha^2}+O\Big(\frac{1}{\alpha^3}\Big) \Big) \Big]-\frac{\alpha}{4}+\sqrt{ \frac{\alpha}{2\pi} }e^{-\frac{\alpha}{8}}
\\
&= 1-e^{-\frac{\alpha}{8}} \sqrt{\frac{8}{\pi \alpha}}  \Big( 1+\frac{\alpha}{4} \Big)  \Big( 1-\frac{4}{\alpha}+\frac{48}{\alpha^2}+O\Big(\frac{1}{\alpha^3}\Big) \Big) +\sqrt{ \frac{\alpha}{2\pi} }e^{-\frac{\alpha}{8}}
\\
&= 1-e^{-\frac{\alpha}{8}} \sqrt{\frac{8}{\pi \alpha}}  \Big( \frac{\alpha}{4}+\frac{8}{\alpha}+O\Big(\frac{1}{\alpha^2}\Big) \Big) +\sqrt{ \frac{\alpha}{2\pi} }e^{-\frac{\alpha}{8}} =  1-e^{-\frac{\alpha}{8}} \sqrt{\frac{8}{\pi \alpha}}  \Big( \frac{8}{\alpha}+O\Big( \frac{1}{\alpha^2} \Big) \Big),
\end{align*}
which leads to the desired behaviour \eqref{c(alpha) fine asym}.

\section{Proof of Theorem \ref{Thm_VN} and Proposition \ref{Prop_s(alpha) zero} }\label{Section_variance}

In this section, we study the variance $V_N$ of the number of real eigenvalues. 
We first show some properties of the function $s(\alpha)$ in \eqref{S(alpha)}. 
In particular, in Lemma~\ref{Lem_S(alpha) alt repre}, we derive an alternative representation of $s(\alpha)$, which will be used in the proof of Theorem~\ref{Thm_VN}, and in Proposition~\ref{Prop_s(alpha) zero}, we derive the asymptotic behaviour of $s(\alpha)$.
The proof of Theorem~\ref{Thm_VN} follows from the analytic expression of $V_N$ given in Lemma~\ref{Lem_SN integral} and the asymptotic behaviour of $b_{j,k}$ given in Proposition~\ref{Prop_ajk asymp}.

\bigskip 

Let us begin with deriving an alternative representation of the function $s(\alpha)$.

\begin{lem} \label{Lem_S(alpha) alt repre}
The function $s(\alpha)$ in \eqref{S(alpha)} has the equivalent representation
\begin{equation} \label{S(alpha) v0}
s(\alpha)=\frac12 \sum_{k=-\infty}^\infty \int_0^1 \Big(\erf \Big( (2k-1)\sqrt{ \frac{\alpha}{8t} } \Big) \Big[   \erf \Big( (2k-1)\sqrt{ \frac{\alpha}{8t} } \Big)  - \erf \Big( (2k+1)\sqrt{ \frac{\alpha}{8t} } \Big) \Big] \Big)\, dt.
\end{equation}
\end{lem}
\begin{proof}
 Since
\begin{equation} \label{int of Gaussian S(alpha)}
\int_{2k+1}^{2k-1} \sqrt{ \frac{\alpha}{8\pi t} } e^{ -\frac{\alpha}{8t}x^2 }\,dx = \frac{1}{2} \Big[   \erf \Big( (2k-1)\sqrt{ \frac{\alpha}{8t} } \Big)  - \erf \Big( (2k+1)\sqrt{ \frac{\alpha}{8t} } \Big) \Big] ,
\end{equation}
the right-hand side of \eqref{S(alpha) v0} can be rewritten as 
\begin{equation}
\begin{split}
s(\alpha)&= \sum_{k=-\infty}^\infty \int_0^1 \erf\Big( (2k-1)\sqrt{ \frac{\alpha}{8t}  } \Big) \int_{2k+1}^{2k-1} \sqrt{ \frac{\alpha}{8\pi t} } e^{ -\frac{\alpha}{8t}x^2 }\,dx \,dt
\\
&= - \sum_{k=-\infty}^\infty \int_0^1 \erf\Big( (2k+1)\sqrt{ \frac{\alpha}{8t}  } \Big) \int_{2k+1}^{2k-1} \sqrt{ \frac{\alpha}{8\pi t} } e^{ -\frac{\alpha}{8t}x^2 }\,dx \,dt,
\end{split}
\end{equation}
after relabelling $k\to -k$ and $x\to-x$. This gives 
\begin{equation}
\begin{split}
s(\alpha)&=\frac12 \sum_{k=-\infty}^\infty \int_0^1 \Big[ \erf\Big( (2k-1)\sqrt{ \frac{\alpha}{8t}  } \Big)- \erf\Big( (2k+1)\sqrt{ \frac{\alpha}{8t}  } \Big)  \Big] \int_{2k+1}^{2k-1} \sqrt{ \frac{\alpha}{8\pi t} } e^{ -\frac{\alpha}{8t}x^2 }\,dx \,dt.
\end{split}
\end{equation}
Using \eqref{int of Gaussian S(alpha)} once again, we obtain that the right-hand side of \eqref{S(alpha) v0} is same as that of \eqref{S(alpha)}. 
\end{proof}

Next, we derive the asymptotic behaviour of $s(\alpha)$. 
We remark that Proposition~\ref{Prop_s(alpha) zero} will not be used in the proof of Theorem~\ref{Thm_VN}; rather as we explained in Section~\ref{Section_main results}, this provides the interpolation \eqref{r(alpha) asymp 0 inf} between \eqref{VN N fixed} and \eqref{VN m fixed}. 

\begin{prop}\label{Prop_s(alpha) zero}
The function $s(\alpha)$ satisfies 
\begin{equation} 
s(\alpha) \sim 
	\begin{cases}
	  \sqrt{ \frac{ \alpha }{\pi} } &\text{as} \quad \alpha \to 0 \, ,
		\\
		1  &\text{as} \quad \alpha \to \infty \, .
	\end{cases}
\end{equation} 
\end{prop}
\begin{proof}

We first show the asymptotic behaviour when $\alpha \to \infty$.
Indeed, this easily follows from the second expression in  \eqref{S(alpha)}, as for $\alpha \to \infty$ only the summand with $k=0$ contributes, 
\begin{equation}
s(\alpha) \sim \int_0^1   \Pr( -1 \le X_{t,\alpha} \le 1 ) ^2 \,dt \sim \int_0^1 1\,dt =1. 
\end{equation}

Next, we show the asymptotic behaviour when $\alpha \to 0$.
By the change of variable $u=\alpha/t$ in \eqref{S(alpha)}, we have  
\begin{equation} \label{S(alpha) change of variable}
s(\alpha)=-\frac{\alpha}{8\pi} \int_\infty^\alpha \frac{1}{u} \sum_{k=-\infty}^\infty \Big(  \int_{2k-1}^{2k+1} e^{-\frac{u}{8}x^2 }\,dx \Big)^2 \,du. 
\end{equation}
Differentiating \eqref{S(alpha) change of variable}, we obtain
\begin{equation} \label{S(alpha) diff eq}
s'(\alpha)=\frac{s(\alpha)}{\alpha}- \frac{1}{8\pi} \sum_{k=-\infty}^\infty \Big(  \int_{2k-1}^{2k+1} e^{-\frac{\alpha}{8}x^2 }\,dx \Big)^2.
\end{equation}

By \eqref{S(alpha) diff eq}, it suffices to show that
\begin{equation} \label{S(alpha) 0 v2}
 \sum_{k=-\infty}^\infty \Big(  \int_{2k-1}^{2k+1} e^{-\frac{\alpha}{8}x^2 }\,dx \Big)^2 \sim  \frac{ 4\sqrt{\pi} }{\sqrt{\alpha}}, \qquad (\alpha \to 0).
\end{equation}
Write $n=1/\sqrt{\alpha}.$ Then we have
\begin{equation} \label{S(alpha) 0 v3}
\begin{split}
\sqrt{\alpha}  \sum_{k=-\infty}^\infty \Big(  \int_{2k-1}^{2k+1} e^{-\frac{\alpha}{8}x^2 }\,dx \Big)^2  & =\frac{1}{n} \sum_{k=-\infty}^\infty \Big(  \int_{2k-1}^{2k+1} e^{-\frac{1}{8} (\frac{x}{n})^2 }\,dx \Big)^2.
\end{split}
\end{equation}
Note that for $k/n=t$, as $n \to \infty$, we have
\begin{equation}
\Big(  \int_{2k-1}^{2k+1} e^{-\frac{1}{8} (\frac{x}{n})^2 }\,dx \Big)^2 = 2n^2 \pi \Big[  \erf\Big( \frac{2nt+1}{2\sqrt{2}n} \Big)-\erf\Big( \frac{2nt-1}{2\sqrt{2}n} \Big)  \Big]^2 \sim 4\,e^{-t^2} .
\end{equation}
Here, we have used the Taylor series of the error function to first order. 
Then, by the Riemann sum approximation, we obtain 
\begin{equation}
\frac{1}{n} \sum_{k=-\infty}^\infty \Big(  \int_{2k-1}^{2k+1} e^{-\frac{1}{8} (\frac{x}{n})^2 }\,dx \Big)^2 \sim 4 \int_{-\infty}^\infty e^{-t^2}\,dt = 4\sqrt{\pi}.
\end{equation}
This gives the desired asymptotic behaviour \eqref{S(alpha) 0 v2}, which completes the proof. 
\end{proof}

Let us now begin to prove Theorem~\ref{Thm_VN}.
It is well known that the variance $V_{N}(m)$ can be expressed in terms of $\bfR_{N,1}^m$ and $\bfR_{N,2}^m$ as
$$ 
V_{N}(m)= \int_{\R} \bfR_{N,1}^m(x)\,dx 
+\int_{\R^2} \Big(\bfR_{N,2}^m(x,y)-\bfR_{N,1}^m(x) \bfR_{N,1}^m(y)\Big)\, dx\,dy, 
$$ 
see e.g. \cite{forrester2007eigenvalue} or \cite[Chapter 16]{Mehta}. 
Due to the Pfaffian structure \eqref{RNkdef} (cf. \eqref{RN1}, \eqref{RN2}), this can be rewritten as 
\begin{equation}
	V_{N}(m)=\int_{\R} S_N(x,x)\,dx 
	-\int_{\R^2}\Big( S_N(x,y)S_N(y,x)+[\text{sgn}(x-y)+\tilde{I}_N(x,y)] D_N(x,y)\Big)\, dx\,dy. 
\end{equation}
Furthermore, integration by parts gives rise to
\begin{equation} \label{VN}
	V_{N}(m)=2E_{N}(m) -2\int_{\R^2} S_N(x,y)S_N(y,x)\,dx\,dy.
\end{equation}
We also refer the reader to \cite{byun2021real} for the use of such an identity in the context of the real elliptic Ginibre ensemble. Notice that this expression is very much reminiscent to the result in determinantal point processes with a single scalar kernel. 

We first evaluate the double integral in \eqref{VN}.
Recall that $b_{j,k}$'s are given by \eqref{bjk ajk}.

\begin{lem} \label{Lem_SN integral}
We have 
\begin{equation} \label{SN integral bjk}
\begin{split}
\int_{\R^2} S_N(x,y) S_N(y,x)\,dx\,dy &=  2  \sum_{\substack{0\le p \le N/2-1 \\ 0 \le q \le N/2-1}} b_{p+1,q+1}\,  b_{q+1,p+1}
\\
&\quad + 2\sum_{\substack{ 0 \le p \le N/2-2 \\ 0 \le q \le N/2-2}} 
b_{q+2,p+1}\,
b_{p+2,q+1} \, 
\\
&\quad -4\sum_{\substack{0\le p \le N/2-2 \\ 0 \le q \le N/2-1}} 
b_{q+1,p+1}\,
b_{p+2,q+1}\,.
\end{split}
\end{equation}
\end{lem}

\begin{proof}
By \eqref{SN}, it holds that
\begin{align*}
 S_N(x,y) S_N(y,x)= \sum_{j,k=0}^{N-2} \frac{\omega(x)\,x^j}{ (2\sqrt{2\pi} j!)^m } (x A_j(y)-A_{j+1}(y)) \frac{\omega(y)\,y^k}{ (2\sqrt{2\pi} k!)^m } (y A_k(x)-A_{k+1}(x)).
\end{align*}
Using \eqref{omega A int ajk}, we have
\begin{align*}
&\quad \int_\R  \omega(x)\,x^j (x A_j(y)-A_{j+1}(y))  (y A_k(x)-A_{k+1}(x)) \,dx
\\
&= -y\,A_j(y) \alpha_{j+2,k+1}+ A_j(y) \alpha_{j+2,k+2}
 +y \,A_{j+1}(y) \alpha_{j+1,k+1} - \,A_{j+1}(y)  \alpha_{j+1,k+2}.
\end{align*}
Applying \eqref{omega A int ajk} once again, we obtain
\begin{align*}
&\quad \int_{\R^2}  \omega(x)\,\omega(y)\,x^j\,y^k  (x A_j(y)-A_{j+1}(y))  (y A_k(x)-A_{k+1}(x)) \,dx\,dy
\\
&=\int_\R  \omega(y)\,\,y^k \Big( A_j(y) \alpha_{j+2,k+2}-y\,A_j(y) \alpha_{j+2,k+1}
 +y \,A_{j+1}(y) \alpha_{j+1,k+1} - \,A_{j+1}(y)  \alpha_{j+1,k+2} \Big) \,dy\\
&=2\Big( \alpha_{j+1,k+1}\alpha_{j+2,k+2}-\alpha_{j+1,k+2} \alpha_{j+2,k+1} \Big),
\end{align*}
upon using the antisymmetry $\alpha_{j,k}=-\alpha_{k,j}$. 
Combining this identity with \eqref{alpha-sym} and further properties \eqref{alpha-sym}, we have 
\begin{align*}
\int_{\R^2} S_N(x,y) S_N(y,x)\,dx\,dy
&=2 \sum_{j,k=0}^{N-2} \frac{ \alpha_{j+1,k+1}\alpha_{j+2,k+2}-\alpha_{j+1,k+2} \alpha_{j+2,k+1}  }{ (2\sqrt{2\pi} j!)^m \,  (2\sqrt{2\pi} k!)^m  } 
\\
&= 2 \sum_{\substack{j+k: odd \\ 0 \le j,k \le N-2}}  \frac{ \alpha_{j+1,k+1}\alpha_{j+2,k+2}  }{ (2\sqrt{2\pi} j!)^m \,  (2\sqrt{2\pi} k!)^m  } - 2 \sum_{\substack{j+k: even \\ 0 \le j,k \le N-2}}  \frac{ \alpha_{j+1,k+2} \alpha_{j+2,k+1}  }{ (2\sqrt{2\pi} j!)^m \,  (2\sqrt{2\pi} k!)^m  } .
\end{align*}
For the first sum we obtain from the symmetry of the summand under exchanging  $k$ and $j$
\begin{align*}
\sum_{\substack{j+k: odd \\ 0 \le j,k \le N-2}}  \frac{ \alpha_{j+1,k+1}\alpha_{j+2,k+2}  }{ (2\sqrt{2\pi} j!)^m \,  (2\sqrt{2\pi} k!)^m  }
&= -2\sum_{\substack{0\le p \le N/2-2 \\ 0 \le q \le N/2-1}}  \frac{ \alpha_{2q+1,2p+2} \, \alpha_{2p+3,2q+2}  }{ (2\sqrt{2\pi} (2p+1)!)^m \,  (2\sqrt{2\pi} (2q)!)^m  }\\
&= -2\sum_{\substack{0\le p \le N/2-2 \\ 0 \le q \le N/2-1}}  \Big( \frac{ 2^{ 2p+2q+1 }\,   }{ \pi (2p+1)! \, (2q)!  }\Big)^m a_{q+1,p+1} \, a_{p+2,q+1}\\
&= -2\sum_{\substack{0\le p \le N/2-2 \\ 0 \le q \le N/2-1}} b_{q+1,p+1} \, b_{p+2,q+1}.
\end{align*}
Here we have used \eqref{alpha a jk} and the identities \eqref{Gamma-id}.
Similarly, we have for the second sum
\begin{align*}
&\quad \sum_{\substack{j+k: even \\ 0 \le j,k \le N-2}}  \frac{ \alpha_{j+1,k+2} \alpha_{j+2,k+1}  }{ (2\sqrt{2\pi} j!)^m \,  (2\sqrt{2\pi} k!)^m  }
\\
&= - \sum_{\substack{0\le p \le N/2-1 \\ 0 \le q \le N/2-1}}  \frac{ \alpha_{2p+1,2q+2} \alpha_{2q+1,2p+2}  }{ (2\sqrt{2\pi} (2p)!)^m \,  (2\sqrt{2\pi} (2q)!)^m  }-\sum_{\substack{ 0 \le p \le N/2-2 \\ 0 \le q \le N/2-2}}  \frac{ \alpha_{2q+3,2p+2} \alpha_{2p+3,2q+2}  }{ (2\sqrt{2\pi} (2p+1)!)^m \,  (2\sqrt{2\pi} (2q+1)!)^m  }\\
&= - \sum_{\substack{0\le p \le N/2-1 \\ 0 \le q \le N/2-1}} \!\Big( \frac{ 2^{2p+2q} }{  \pi (2p)!(2q)!  }\Big)^m a_{p+1,q+1} \, a_{q+1,p+1}
-\sum_{\substack{ 0 \le p \le N/2-2 \\ 0 \le q \le N/2-2}}  \!\Big(\frac{ 2^{2p+2q+2}  }{ \pi (2p+1)!\,  (2q+1)!  } \Big)^m a_{q+2,p+1}\, a_{p+2,q+1}\\
&= - \sum_{\substack{0\le p \le N/2-1 \\ 0 \le q \le N/2-1}} b_{p+1,q+1} \, b_{q+1,p+1}-\sum_{\substack{ 0 \le p \le N/2-2 \\ 0 \le q \le N/2-2}}  b_{q+2,p+1} \, b_{p+2,q+1} .
\end{align*}
Combining all of the above, the desired equation \eqref{SN integral bjk} follows.
\end{proof}

We are now ready to prove Theorem~\ref{Thm_VN}.

\begin{proof}[Proof of Theorem~\ref{Thm_VN}]

By \eqref{VN} and Theorem~\ref{Thm_EN}, we have
\begin{equation}
\begin{split}
	\frac{V_{N}(\alpha N)}{E_N(\alpha N)} & =2 - \frac{2}{E_N(\alpha N)} \int_{\R^2} S_N(x,y)S_N(y,x)\,dx\,dy
	\\
	&\sim 2- \frac{2}{N c(\alpha)}   \int_{\R^2} S_N(x,y)S_N(y,x)\,dx\,dy.
\end{split}
\end{equation}
Therefore by Lemma~\ref{Lem_S(alpha) alt repre}, it suffices to show that 
\begin{equation} \label{SN int lim}
\lim_{N \to \infty} \frac{1}{N} \int_{\R^2} S_N(x,y)S_N(y,x)\,dx\,dy= s(\alpha),
\end{equation}
where $s(\alpha)$ is given by \eqref{S(alpha) v0}.

Note that by Proposition~\ref{Prop_ajk asymp}, we have that for $l \in \mathbb{Z}$ fixed,
\begin{equation} \label{b j j+l}
b_{j,j+l} \sim \frac12 \Big[ 1+\erf \Big( (2l+1)\sqrt{ \frac{\alpha}{8t} } \Big)  \Big].
\end{equation}
It order to evaluate the double sums in Lemma \ref{Lem_SN integral} we rearrange them in the following way:
\begin{equation}
\sum_{\substack{0\le p \le L \\ 0 \le q \le L}} A_{p,q}=
\sum_{l=-L}^L\sum_{k+l=0}^L A_{k,k+l}, 
\qquad \sum_{\substack{0\le p \le L -1\\ 0 \le q \le L}} A_{p,q}=
\sum_{l=-L}^L\sum_{j+l=0}^{L-1} A_{j+l,j},
\end{equation}
for the first and second, respectively third sum. 
Using the asymptotic behaviour \eqref{b j j+l}, we thus obtain for 
the first sum in \eqref{SN integral bjk} that as $N \to \infty$, 
\begin{align*}
\frac{2}{N}
\sum_{\substack{0\le p \le N/2-1 \\ 0 \le q \le N/2-1}} 
b_{q+1,p+1} &= \sum_{l =-N/2+1}^{N/2-1} \,\frac{2}{N} \sum_{k+l=0}^{N/2-1}  b_{k+1,j+l+1}\,  b_{j+l+1,j+1} 
\\
&\sim  \sum_{l \in \mathbb{Z}} \frac14\int_0^1 \Big[ 1+\erf \Big( (2l+1)\sqrt{ \frac{\alpha}{8t} } \Big)  \Big]\cdot \Big[ 1-\erf \Big( (2l-1)\sqrt{ \frac{\alpha}{8t} } \Big)  \Big] \,dt .
\end{align*}
Here we have used the Riemann sum approximation. 
Similarly, we have for the second sum
\begin{align*}
\frac{2}{N}\sum_{\substack{0\le p \le N/2-2 \\ 0 \le q \le N/2-2}} 
 b_{p+2,q+1} \, b_{q+2,p+1} &=  \sum_{l =-N/2+2}^{N/2-2} \frac{2}{N} \sum_{j+l=0}^{N/2-2} b_{j+2,j+l+1} \, b_{j+l+2,j+1} 
\\
&\sim    \sum_{l \in \mathbb{Z}} \frac14 \int_0^1  \Big[ 1+\erf \Big( (2l-1)\sqrt{ \frac{\alpha}{8t} } \Big)  \Big] \cdot \Big[ 1-\erf \Big( (2l+1)\sqrt{ \frac{\alpha}{8t} } \Big)  \Big]\,dt
\end{align*}
and for the third
\begin{align*}
\frac{4}{N}
\sum_{\substack{0\le p \le N/2 -2\\ 0 \le q \le N/2-1}} 
b_{p+2,q+1} \, b_{q+1,p+1} &=   \sum_{l =-N/2+1}^{N/2-1} \frac{4}{N} \sum_{j+l=0}^{N/2-2} b_{j+2,j+l+1} \, b_{j+l+1,j+1} 
\\
&\sim \sum_{l \in \mathbb{Z}} \frac12 \int_0^1  \Big[ 1-\erf \Big( (2l-1)\sqrt{ \frac{\alpha}{8t} } \Big)^2  \Big]\,dt.
\end{align*}
Combining all of the above equations, we obtain \eqref{SN int lim}. 
This completes the proof.
\end{proof}


\section{Proof of Theorem \ref{Thm_density}}\label{Section_density}

In this section, we study densities of real eigenvalues.
To analyse the limiting densities, we shall use the method of moments following the idea by Simm \cite{MR3685239}. 
Namely, we consider the moments \eqref{moment unscaled} and \eqref{moment rescaling} for real eigenvalues and study their large-$N$ behaviour to derive limiting densities. 
To perform the asymptotic analysis, we use the analytic expressions of the moments (Lemmas~\ref{Lem_moment ajk Simm} and \ref{Lem_moment ajk rescale}).  

\bigskip

Let us first show the first part of Theorem~\ref{Thm_density}. 
For this purpose, let 
\begin{equation} \label{moment unscaled}
M_{k,N}(m):= \int_\R x^k \, \bfR_{N,1}^m(x)\,dx
\end{equation}
be the $k$-th the moment of the $1$-point function. 
Note that $M_{0,N}(m)=E_N(m)$ and that the 1-point function is even.

Recall the following property from \cite[Lemma 2.3]{MR3685239}. 
This can be shown using Lemma~\ref{Lem_int of omega} and \eqref{SN}.

\begin{lem} \label{Lem_moment ajk Simm}
We have $M_{2k+1,N}(m)=0$ and
\begin{equation}
M_{2k,N}(m)= M_{2k,N}^{(1)}(m)-M_{2k,N}^{(2)}(m),
\end{equation}
where 
\begin{align}
M_{2k,N}^{(1)}(m)&:= \Big(\frac{2}{N}\Big)^{mk} \sum_{j=0}^{N/2-1}  \frac{ a_{j+1,j+k+1}+a_{j+k+1,j+1} }{ (\Gamma(j+\frac12)\Gamma(j+1) )^m }, \label{M2kN 1}
\\
M_{2k,N}^{(2)}(m)&:= \Big(\frac{2}{N}\Big)^{mk} \sum_{j=0}^{N/2-2}  \frac{ a_{j+2,j+k+1}+a_{j+k+2,j+1} }{ (\Gamma(j+\frac32)\Gamma(j+1) )^m }. \label{M2kN 2}
\end{align}
\end{lem}

\begin{rmk*}
We remark that Lemma~\ref{Lem_moment ajk Simm} was used in \cite{MR3685239} to derive the limiting density \eqref{density m fixed}. 
To be more precise, it was shown that  
\begin{equation}
\lim_{N \to \infty} \frac{ M_{k,N}(m) }{ M_{0,N}(m) }= 
\begin{cases}
(mk+1)^{-1} &k \textup{ even},
\\
0 &k \textup{ odd}. 
\end{cases}
\end{equation}
Then \eqref{density m fixed} follows from the fact that
\begin{equation}
\int_{-1}^1 x^k\,\rho^m(x)\,dx = 
\begin{cases}
(mk+1)^{-1} &k \textup{ even},
\\
0 &k \textup{ odd}. 
\end{cases}
\end{equation}
\end{rmk*}

To analyse the moment $M_{k,N}(\alpha N)$, we need the following lemma. 

\begin{lem} \label{Lem_ajk asymp moment}
Let $j=tN/2$ with $t \in (0,1)$ and $m=\alpha N$. 
Then for any fixed $k$, we have
\begin{equation}
\Big(\frac{2}{N}\Big)^{mk} \, \frac{ a_{j+1,j+k+1}+a_{j+k+1,j+1} }{ (\Gamma(j+\frac12)\Gamma(j+1) )^m } =O(t^{k\alpha N})
\end{equation}
and 
\begin{equation}
\Big(\frac{2}{N}\Big)^{mk} \, \frac{ a_{j+2,j+k+1}+a_{j+k+2,j+1} }{ (\Gamma(j+\frac32)\Gamma(j+1) )^m }= O(t^{k\alpha N})
\end{equation}
as $N \to \infty$.
\end{lem}

\begin{proof}

Note that by \eqref{bjk ajk},
\begin{equation}
 \frac{ a_{j+1,j+k+1} }{ (\Gamma(j+\frac12)\Gamma(j+1) )^m } =  \Big( \frac{\Gamma(j+k+1)}{ \Gamma(j+1) }  \Big)^m  b_{j+1,k+1}.
\end{equation}
Since 
\begin{equation}
\Big( \frac{2}{N} \Big)^{k} \frac{ \Gamma(j+k+1) }{ \Gamma(j+1) } \sim \Big( \frac{2}{N} j \Big)^k \sim t^k, 
\end{equation}
it follows from Proposition~\ref{Prop_ajk asymp} that
\begin{equation}
\begin{split}
\Big(\frac{2}{N}\Big)^{mk}\frac{a_{j+1,j+k+1}}{( \Gamma(j+\frac12) \Gamma(j+1) )^m} &=   \Big(  \Big(\frac{2}{N}\Big)^{k} \frac{\Gamma(j+k+1)}{\Gamma(j+1)} \Big)^m  b_{j+1,j+k+1}  =O(t^{k\alpha N}).
\end{split}
\end{equation}
Similarly, by combining 
\begin{equation}
 \Big( \frac{2}{N} \Big)^k \frac{ \Gamma(j+k+\frac12) }{ \Gamma(j+\frac12) } \sim \Big( \frac{2}{N} j \Big)^k  \sim t^k , \qquad  
 \Big( \frac{2}{N} \Big)^k \frac{ \Gamma(j+k+\frac32) }{ \Gamma(j+\frac32) } \sim \Big( \frac{2}{N} j \Big)^k  \sim t^k
\end{equation}
with Proposition~\ref{Prop_ajk asymp}, we obtain 
\begin{align}
\Big(\frac{2}{N}\Big)^{mk}\frac{a_{j+k+1,j+1}}{( \Gamma(j+\frac12) \Gamma(j+1) )^m} & =O(t^{k\alpha N}),
\\
\Big(\frac{2}{N}\Big)^{mk}\frac{a_{j+2,j+k+1}}{( \Gamma(j+\frac32) \Gamma(j+1) )^m} & =O(t^{k\alpha N}),
\\
\Big(\frac{2}{N}\Big)^{mk}\frac{a_{j+k+2,j+1}}{( \Gamma(j+\frac32) \Gamma(j+1) )^m} & =O(t^{k\alpha N}).
\end{align}
This completes the proof. 
\end{proof}

We are now ready to show the first part of Theorem~\ref{Thm_density}.

\begin{proof}[Proof of Theorem~\ref{Thm_density} (i)]

By combining Lemma~\ref{Lem_ajk asymp moment} with \eqref{M2kN 1} and \eqref{M2kN 2}, the Riemann sum approximation gives that 
\begin{equation}
M_{2k,N}^{(1)}(\alpha N)=o(1), \qquad M_{2k,N}^{(2)}(\alpha N)=o(1).
\end{equation}
Therefore by Lemma~\ref{Lem_moment ajk Simm}, we obtain that for each $k \ge 1$, 
\begin{equation}
\lim_{N \to \infty} M_{k,N}(\alpha N)= 0.
\end{equation}
Since all higher moments vanish, it gives the characterising feature of the Dirac delta measure. 
This completes the proof. 
\end{proof}

We now prove the second part of Theorem~\ref{Thm_density}. 
Recall that we use the rescaling \eqref{Lyapunov resacling}.
Then by the change of variable, the $1$-point function $\wt{\bfR}_{N,1}^m$ for the rescaled variable $\lambda$ is given by 
\begin{equation}
\wt{\bfR}_{N,1}^m(\lambda)=m\,\lambda^{m-1}\,\bfR_{N,1}^m(\lambda^m).
\end{equation}
Therefore the associated moment $\wt{M}_{k,N}(m)$ for $k$ even is given by 
\begin{equation} \label{moment rescaling}
\wt{M}_{k,N}(m) =m\int_\R \lambda^{k+m-1}\,\bfR_{N,1}^m(\lambda^m)\,d\lambda.
\end{equation}

Along the same lines of the proof of Lemma~\ref{Lem_moment ajk Simm}, we have the following.

\begin{lem} \label{Lem_moment ajk rescale}
We have $\wt{M}_{2k+1,N}(m)=0$ and
\begin{equation}
\wt{M}_{2k,N}(m)= \wt{M}_{2k,N}^{(1)}(m)-\wt{M}_{2k,N}^{(2)}(m),
\end{equation}
where 
\begin{align}
\wt{M}_{2k,N}^{(1)}(m)&:= \Big(\frac{2}{N}\Big)^{k} \sum_{j=0}^{N/2-1}  \frac{ a_{j+1,j+\frac{k}{m}+1}+a_{j+\frac{k}{m}+1,j+1} }{ (\Gamma(j+\frac12)\Gamma(j+1) )^m }, \label{wt M2kN 1}
\\
\wt{M}_{2k,N}^{(2)}(m)&:= \Big(\frac{2}{N}\Big)^{k} \sum_{j=0}^{N/2-1}  \frac{ a_{j+2,j+\frac{k}{m}+1}+a_{j+\frac{k}{m}+2,j+1} }{ (\Gamma(j+\frac32)\Gamma(j+1) )^m }. \label{wt M2kN 2}
\end{align}
\end{lem}

We now analyse the moments \eqref{moment rescaling}. The key ingredient is the following lemma. 

\begin{lem}\label{Lem_ajk asymptotic moment rescale}
Let $j=tN/2$ with $t \in (0,1)$, $m=\alpha N$, and $k \in \mathbb{N}$ be fixed.
Then, as $N \to \infty$, we have 
\begin{equation}
\Big(\frac{2}{N}\Big)^{k} \, \frac{ a_{j+1,j+\frac{k}{m}+1}+a_{j+\frac{k}{m}+1,j+1} }{ (\Gamma(j+\frac12)\Gamma(j+1) )^m } \sim t^k \Big[ 1+\erf \Big( \sqrt{ \frac{\alpha}{8t} } \Big)  \Big]
\end{equation}
and 
\begin{equation}
\Big(\frac{2}{N}\Big)^{k} \, \frac{ a_{j+2,j+\frac{k}{m}+1}+a_{j+\frac{k}{m}+2,j+1} }{ (\Gamma(j+\frac32)\Gamma(j+1) )^m } \sim t^k \Big[ 1-\erf \Big( \sqrt{ \frac{\alpha}{8t} } \Big)  \Big].
\end{equation}
\end{lem}
\begin{proof}
Following the proof of Proposition~\ref{Prop_ajk asymp}, we move from coefficients $a_{j,k}$ to $b_{j,k}$ according to \eqref{bjk ajk}, and use the contour integral representation of the latter. 
For instance, by \eqref{ajk Gamma contour}, we have 
\begin{equation}
b_{j+1,j+\frac{k}{m}+1} =-\frac{1}{2\pi i} \int_C \Big(  \frac{\Gamma(j+\frac12+s) }{ \Gamma(j+\frac12)  }  \frac{ \Gamma(j+\frac{k}{m}+1-s)}{ \Gamma(j+\frac{k}{m}+1) }  \Big)^m\,\frac{ds}{s}. 
\end{equation}
Then by \eqref{Gamma ratio general}, 
\begin{equation}
\Big(  \frac{\Gamma(j+\frac12+s) }{ \Gamma(j+\frac12)  }  \frac{ \Gamma(j+\frac{k}{m}+1-s)}{ \Gamma(j+\frac{k}{m}+1) }  \Big)^m \sim  e^{ \frac{\alpha}{t} s(2s-1) }.
\end{equation}
 The remaining computation for the asymptotic behaviour of the first coefficient in \eqref{wt M2kN 1} 
is same as the one in the proof of Proposition~\ref{Prop_ajk asymp}, and similarly for the second one:
\begin{equation} \label{b j km 12}
b_{j+1,j+\frac{k}{m}+1} \sim \frac12 \Big[ 1+\erf\Big( \sqrt{ \frac{\alpha}{8t} } \Big) \Big], \qquad  b_{j+\frac{k}{m}+1,j+1} \sim \frac12 \Big[ 1+\erf\Big( \sqrt{ \frac{\alpha}{8t} } \Big) \Big].
\end{equation}
Along the same lines we also have for \eqref{wt M2kN 2}
\begin{equation} \label{b j km 34}
b_{j+2,j+\frac{k}{m}+1} \sim \frac12 \Big[ 1-\erf\Big( \sqrt{ \frac{\alpha}{8t} } \Big) \Big], \qquad  b_{j+\frac{k}{m}+2,j+1} \sim \frac12 \Big[ 1-\erf\Big( \sqrt{ \frac{\alpha}{8t} } \Big) \Big].
\end{equation}
The remaining proportionality factors when moving from $a_{j,k}$ to $b_{j,k}$ lead to the following contributions.
By \eqref{Gamma duplication} we have 
\begin{equation}
\frac{ \Gamma(j+\frac{k}{m}+1) }{ \Gamma(j+1) }  = j^{ \frac{k}{m} } \Big(  1+\frac{ \frac{k}{m}(1+\frac{k}{m}) }{2j}+O(j^{-2}) \Big)=j^{ \frac{k}{m} } \Big( 1+O(N^{-2}) \Big),
\end{equation}
which leads to 
\begin{equation}
\Big( \frac{2}{N} \Big)^k \Big(  \frac{ \Gamma(j+\frac{k}{m}+1) }{ \Gamma(j+1) }   \Big)^m \sim t^k. 
\end{equation}
Similarly we have 
\begin{equation}
\Big( \frac{2}{N} \Big)^k \Big(  \frac{ \Gamma(j+\frac{k}{m}+\frac12) }{ \Gamma(j+\frac12) }   \Big)^m \sim t^k, \qquad \Big( \frac{2}{N} \Big)^k \Big(  \frac{ \Gamma(j+\frac{k}{m}+\frac32) }{ \Gamma(j+\frac32) }   \Big)^m \sim t^k.
\end{equation}
Combining all of the above the proof is complete. 
\end{proof}

We are now ready to prove the second assertion of Theorem~\ref{Thm_density}.

\begin{proof}[Proof of Theorem~\ref{Thm_density} (ii)]
Let $k$ be a positive even integer. By combining Lemma~\ref{Lem_ajk asymptotic moment rescale} with \eqref{wt M2kN 1} and \eqref{wt M2kN 2}, we have 
\begin{equation}
\frac{\wt{M}_{2k,N}^{(1)}(\alpha N)}{N} \sim   \frac12 \int_0^1 t^k \Big[ 1+\erf \Big( \sqrt{ \frac{\alpha}{8t} } \Big)\Big] \,dt, \qquad 
\frac{\wt{M}_{2k,N}^{(2)}(\alpha N)}{N} \sim    \frac12 \int_0^1 t^k \Big[ 1-\erf \Big( \sqrt{ \frac{\alpha}{8t} } \Big)\Big] \,dt.
\end{equation}
By definition, this gives rise to 
\begin{equation}
\lim_{N \to \infty} \frac{ \wt{M}_{k,N}(\alpha N) }{ N } =  \int_0^1 t^k \, \erf \Big( \sqrt{ \frac{\alpha}{8t} } \Big)\,dt.
\end{equation}
By Theorem~\ref{Thm_EN} and the change of variable $t=\lambda^2$, we obtain 
\begin{equation}
\begin{split}
\lim_{N \to \infty} \frac{ \wt{M}_{2k,N}(\alpha N) }{ E_N(\alpha N) } & = \frac{1}{c(\alpha)} \int_0^1 t^k \, \erf \Big( \sqrt{ \frac{\alpha}{8t} } \Big)\,dt 
\\
&= \frac{1}{c(\alpha)} \int_{-1}^1 \lambda^{2k} \, |\lambda|\, \erf\Big( \sqrt{ \frac{\alpha}{8} } \frac{1}{|\lambda|}  \Big) \,d\lambda = \int_\R \lambda^{2k} \, \wt{\rho}_\alpha(\lambda)\,d\lambda, 
\end{split}
\end{equation}
which completes the proof. 
\end{proof}

\subsection*{Acknowledgements} We gratefully acknowledge Nam-Gyu Kang for several helpful comments concerning Proposition~\ref{Prop_s(alpha) zero} and for the much-appreciated help in improving this manuscript.
We also thank Nick Simm for his interest in this work and him as well as Jesper Ipsen and Peter Forrester for their comments.

\bibliographystyle{abbrv}
\bibliography{RMTbib}
\end{document}